\documentclass[11pt]{amsart}
\usepackage[utf8]{inputenc}
\usepackage{amsmath,amsthm,comment}
\usepackage{amsfonts,enumitem}
\usepackage{amssymb}

\usepackage{graphicx,xy,pstricks}
\input xy
\xyoption{all}
\newcommand{\executeiffilenewer}[3]{%
 \ifnum\pdfstrcmp{\pdffilemoddate{#1}}%
 {\pdffilemoddate{#2}}>0%
 {\immediate\write18{#3}}\fi%
}

\theoremstyle{definition} 

 \newtheorem{definition}{Definition}[section]
 
 \newtheorem{example}[definition]{Example}

\theoremstyle{plain}      

 \newtheorem{proposition}[definition]{Proposition}
 \newtheorem{theorem}[definition]{Theorem}
 
 \newtheorem{lemma}[definition]{Lemma}

\newtheorem*{theorem*}{Theorem}

\newcommand{\R}{\mathbb{R}}
\newcommand{\C}{\mathbb{C}}
\newcommand{\Q}{\mathbb{Q}}

\newcommand{\T}{\mathbb{T}}
\newcommand{\Z}{\mathbb{Z}}

\renewcommand{\P}{\mathbb{P}}

\DeclareMathOperator{\tr}{Tr}

\DeclareMathOperator{\SL}{SL}

\DeclareMathOperator{\SU}{SU}
\DeclareMathOperator{\su}{su}
\DeclareMathOperator{\CS}{CS}

\DeclareMathOperator{\Hom}{Hom}

\DeclareMathOperator{\Map}{Map}

\DeclareMathOperator{\Res}{Res}
\DeclareMathOperator{\Mod}{Mod}
\DeclareMathOperator{\coker}{Coker}
\DeclareMathOperator{\Id}{Id}

\title{Distribution of Chern-Simons invariants}
\author{Julien March\'e}
\date{} 
\address{Institut de Math\'ematiques de Jussieu - Paris Rive Gauche, Universit\'e Pierre et Marie Curie, 75252 Paris c\'edex 05, France}
\email{julien.marche@imj-prg.fr}

\begin{document}
\date{}
\maketitle

\begin{abstract}
Let $M$ be a 3-manifold with a finite set $X(M)$ of conjugacy classes of representations $\rho:\pi_1(M)\to\SU_2$. We study here the distribution of the values of the Chern-Simons function $\CS:X(M)\to \R/2\pi\Z$. We observe in some examples that it resembles the distribution of quadratic residues. In particular for specific sequences of $3$-manifolds, the invariants tends to become equidistributed on the circle with white noise fluctuations of order $|X(M)|^{-1/2}$. We prove that for a manifold with toric boundary the Chern-Simons invariants of the Dehn fillings $M_{p/q}$ have the same behaviour when $p$ and $q$ go to infinity and compute fluctuations at first order. 
\end{abstract}

\section{Introduction}
\subsection{Distribution of quadratic residues}
Let $p$ be a prime number. We consider the weighted counting measure on the circle $\T=\R/2\pi\Z$ defined by quadratic residues modulo $p$, that is: 
$$\mu_p=\frac{1}{p}\sum_{k=0}^{p-1}\delta_{\frac{2\pi k^2}{p}}.$$
We investigate the limit of $\mu_p$ when $p$ goes to infinity and to that purpose, we consider its $\ell$-th momentum i.e $\mu_p^\ell=\int e^{i\ell\theta} d\mu_p(\theta)=\frac{1}{p}\sum_{k=0}^{p-1}\exp(2i\pi \ell k^2/p)$. We have $\mu_p^\ell=1$ if $p|\ell$, and else by the Gauss sum formula,  $\mu_p^\ell=\genfrac(){}{}{\ell}{p}\frac{1}{\sqrt{p}}$ where $\genfrac(){}{}{\ell}{p}$ is the Legendre symbol. 

This shows that $\mu_p$ converges to the uniform measure $\mu_\infty$ whereas the renormalized measure $\sqrt{p}(\mu_p-\mu_{\infty})$ -that we call fluctuation- has $l$-th momentum $\pm1$ depending on the residue of $l$ modulo $p$ and hence is a kind of ``white noise".

\subsection{Distribution of Chern-Simons invariants}

On the other hand, such Gauss sums appear naturally in the context of Chern-Simons invariants of 3-manifolds. Consider an oriented and compact 3-manifold $M$ and define its character variety as the set $X(M)=\Hom(\pi_1(M),\SU_2)/\SU_2$. In what follows, we will confuse between representations and their conjugacy classes. The Chern-Simons invariant may be viewed as a locally constant map $\CS:X(M)\to \T$. We refer to \cite{kirk-klassen} for background on Chern-Simons invariants and give here a quick definition for the convenience of the reader. 

Let $\nu$ be the Haar measure of $\SU_2$ normalised by $\nu(\SU_2)=2\pi$ and let $\pi:\tilde{M}\to M$ be the universal cover of $M$. There is an equivariant map $F:\tilde{M}\to \SU_2$ in the sense that $F(\gamma x)=\rho(\gamma)F(x)$ for all $\gamma\in \pi_1(M)$ and $x\in M$. The form $F^*\nu$ is invariant hence can be written $F^*\nu=\pi^*\nu_F$. We set $\CS(\rho)=\int_M \nu_F$ and claim that it is independent on the choice of equivariant map $F$ modulo $2\pi$. 

\begin{definition}
Let $M$ be a 3-manifold whose character variety is finite. We define its {\it Chern-Simons measure} as 
$\mu_M=\frac{1}{|X(M)|}\sum\limits_{\rho\in X(M)}\delta_{\CS(\rho)}$. 
\end{definition}
\subsubsection{Lens spaces}

For instance, if $M=L(p,q)$ is a lens space, then $\pi_1(M)=\Z/p\Z$ and $X(M)=\{\rho_n,n\in \Z/p\Z\}$ where $\rho_n$ maps the generator of $\Z/p\Z$ to a matrix with eigenvalues $e^{\pm \frac{2i\pi n}{p}}$. We know from \cite{kirk-klassen} that $\CS(\rho_n)=2\pi \frac{q^*n^2}{p}$ where $qq^*=1 \bmod p$. Hence, the Chern-Simons invariants  of $L(p,q)$ behave exactly like quadratic residues when $p$ goes to infinity. 

\subsubsection{Brieskorn spheres}
To give a more complicated but still manageable example, consider the Brieskorn sphere $M=\Sigma(p_1,p_2,p_3)$ where $p_1,p_2,p_3$ are distinct primes. This is a homology sphere whose irreducible representations in $\SU_2$ have the form $\rho_{n_1,n_2,n_3}$ where $0<n_1<p_1,0<n_2<p_2,0<n_3<p_3$. 
From \cite{kirk-klassen} we have $$\CS(\rho_{n_1,n_2,n_3})=2\pi\frac{(n_1p_2p_3+p_1n_2p_3+p_1p_2n_3)^2}{4p_1p_2p_3}$$

Setting $n=n_1p_2p_3+p_1n_2p_3+p_1p_2n_3$, we observe that -due to Chinese remainder theorem- $n$ describes $(\Z/p_1p_2p_3\Z)^\times$ when $n_i$ describes $(\Z/p_i\Z)^\times$ for $i=1,2,3$. 
Hence, we compute that the following $\ell$-th momentum: 
$$\mu_{p_1p_2p_3}^\ell=\frac{1}{|X(\Sigma(p_1,p_2,p_3))|}\sum_{\rho\in X(M)}\exp(i\ell\CS(\rho))\sim \frac{1}{p_1p_2p_3}\sum_{n=0}^{p_1p_2p_3-1}e^{\frac{i\pi \ell n^2}{2p_1p_2p_3}}.$$

Assuming $\ell$ is coprime with $p=p_1p_2p_3$ we get from \cite{bew} the following estimates where $\epsilon_n=1$ is $n=1\bmod 4$ and $\epsilon_n=i$ if $n=3\bmod 4$:
$$\mu_{p}^{\ell}\sim \begin{cases} \frac{\epsilon_p}{\sqrt{p}}\genfrac(){}{}{\ell/4}{p}\quad\text{ if }\ell =0\bmod 4\\ 0\quad\text{ if }\ell=2\bmod 4 \\ \frac{1+i}{2\sqrt{p}\epsilon_l}\genfrac(){}{}{p}{\ell} \quad\text{ else.}\end{cases}$$
Again we obtain that $\mu_p$ converges to the uniform measure when $p$ goes to infinity. The renormalised measure $\sqrt{p}(\mu_p-\mu_\infty)$ have $\ell$-th momentum with modulus equal to $1,\frac{1}{\sqrt{2}},0,\frac{1}{\sqrt{2}}$ depending on $\ell\bmod 4$.

\subsection{Dehn Fillings}
The main question we address in this article is the following: fix a manifold $M$ with boundary $\partial M=\T\times \T$. For any $\frac{p}{q}\in \P^1(\Q)$, we denote by $\T_{p/q}$ the curve on $\T^2$ parametrised by $(pt, qt)$ for $t$ in $\T$. We define the manifold $M_{p/q}$ by Dehn filling i.e the result of gluing $M$ with a solid torus such that $\T_{p/q}$ bounds a disc.

We recall from \cite{kirk-klassen} that in the case where $M$ has boundary, there is a principal $\T$-bundle with connection $L\to X(\partial M)$ such that the Chern-Simons invariant is a flat section of $\Res^*L$

$$\xymatrix{ & L\ar[d] \\ X(M)\ar[ur]^{\CS}\ar[r]^{\Res} & X(\partial M)}$$

where $\Res(\rho)=\rho\circ i_*$ and $i:\partial M\to M$ is the inclusion. 

We will denote by $|d\theta|$ the natural density on $X(\T)=\T/(\theta\sim -\theta)$. 

We also have $X(\T^2)=\T^2/(x,y)\sim (-x,-y)$ and for any $p,q$ the map $\Res_{p/q}: X(\T^2)\to X(\T_{p/q})$ is given by $(x,y)\mapsto px+qy$.

Moreover, for any $\frac{p}{q}$, $\ell>0$ and $0\le k\le \ell$, there are natural flat sections $\CS_{p/q}^{k/\ell}$ of $L^\ell$ over the preimage $\Res_{p/q}^{-1}(\frac{\pi k}{\ell})$. These sections are called Bohr-Sommerfeld sections and they coincide for $k=0$ with $\CS^\ell$. See \cite{kirk-klassen} or \cite{charles-marche} for a detailed description. 

\begin{theorem}\label{main}
Let $M$ be a 3-manifold with $\partial M=\T^2$ satisfying the hypothesis of Section \ref{hypothesis}. Let $p,q,r,s$ be integers satisfying $ps-qr=1$ and for any integer $n$, set $p_n=pn-r$ and $q_n=qn-s$. Then setting 
$$\mu_n^\ell=\frac{1}{n}\sum_{\rho\in X(M_{p_n/q_n})}e^{i\ell \CS(\rho)}$$
we get first 
$$\mu_n^0= \int_{X(M)}\Res_{r/s}^* |d\theta|+O\Big(\frac{1}{n}\Big)$$
and for $\ell>0$ 
$$\mu_n^\ell=\frac{1}{\sqrt{2n}}\sum_{k=0}^{l} \sum_{\rho,k/ \Res_{r/s}(\rho)=\pi\frac{k}{l}}\exp(-2i\pi n \frac{k^2}{4\ell}+i\ell \CS(\rho)-i\CS_{r/s}^{k/l}(\rho))+O(\frac{1}{n})$$
\end{theorem}

Hence, we recover the behaviour that we observed for Lens spaces and Brieskorn spheres. The measure converges to a uniform measure $\mu_\infty$ and the renormalised measure $\sqrt{n}(\mu_n-\mu_\infty)$ has an oscillating behaviour controlled by representations in $X(M)$ with rational angle along $\T_{r/s}$. 

\subsection{Intersection of Legendrian subvarieties}

We will prove Theorem \ref{main} in the more general situation of curves immersed in a torus. Indeed, the problem makes sense in an even more general setting that we present here.

\subsubsection{Prequantum bundles}
\begin{definition}
Let $(M,\omega)$ be a symplectic manifold. A prequantum bundle is a principal $\T$-bundle with connection whose curvature is $\omega$. 
\end{definition}

It is well-known that the set of isomorphism classes of prequantum bundles is homogeneous under $H^1(M,\T)$ and non-empty if and only if $\omega$ vanishes in $H^2(M,\T)$. Let us give three examples:
 


\begin{example}\label{exemple}
\begin{enumerate}[label=(\roman*)]

\item Take $\R^2\times\T$ with $\lambda=d\theta+\frac{1}{2\pi}(xdy-ydx)$. This gives a prequantum bundle on $\R^2$. Dividing by the action of $\Z^2$ given by 
\begin{equation}\label{action}
(m,n)\cdot(x,y,\theta)=(x+2\pi m,y+2\pi n,\theta+my-nx)
\end{equation}
gives a prequantum bundle $\pi:L\to\T^2$. 

\item Any complex projective manifold $M\subset \P^n(\C)$ has such a structure by restricting the tautological bundle whose curvature is the restriction of the Fubini-Study metric. 

\item The Chern-Simons bundle over the character variety of a surface.

\end{enumerate}
\end{example}
In all these cases, there is a natural subgroup of the group of symplectomorphisms of $(M,\omega)$ which acts on the prequantum bundle. The group $\SL_2(\Z)$ acts in the first case and the mapping class group in the third case. In the second case, a group acting linearly on $\C^{n+1}$ and preserving $M$ will give an example. 
%
%
\subsubsection{Legendrian submanifolds and their pairing}
Consider a prequantum bundle $\pi:L\to M$ where $M$ has dimension $2n$ and denote by $\lambda\in\Omega^1(L)$ the connection 1-form. By Legendrian immersion we will mean an immersion $i:N\to L$ where $N$ is a manifold of dimension $n+1$ such that $i^*\lambda=0$. This condition implies that $i$ is transverse to the fibres of $\pi$ and hence $\pi\circ i:N\to M$ is a Lagrangian immersion. 

\begin{definition}\begin{enumerate}
\item
Given $i_1:N_1\to L$ and $i_2:N_2\to L$ two Legrendrian immersions, we will say that they are transverse if it is the case of $\pi\circ i_1$ and $\pi\circ i_2$. 
\item 
Given such transverse Legendrian immersions and an intersection point, i.e. $x_1\in N_1$ and $x_2\in N_2$ such that $\pi(i_1(x_1))=\pi(i_2(x_2))$ we define their phase $\phi(i_1(x_1),i_2(x_2))$ as the element $\theta\in \T$ such that $i_2(x_2)=i_1(x_1)+\theta$. 
\item The phase measure $\phi(i_1,i_2)$ is the measure on the circle defined by 
$$ \phi(i_1,i_2) = \sum_{\pi(i_1(x_1))=\pi(i_2(x_2))}\delta_{\phi(i_1(x_1),i_2(x_2))}.$$
\end{enumerate}
\end{definition}
If $M$ is a 3-manifold obtained as $M=M_1\cup M_2$ then, assuming transversality, the Chern-Simons measure of $M$ is given by $\mu_M=\phi(\CS_1,\CS_2)$ where $\CS_i:X(M_i)\to L$ is the Chern-Simons invariant with values in the Chern-Simons bundle. 

\section{The torus case}

\subsection{Immersed curves in the torus}
Consider the pre quantum bundle $\pi:L\to\T^2$ given in the first item of Example \ref{exemple}. We consider a fixed Legendrian immersion $i:[a,b]\to L$ and for any coprime integers $p,q$ the Legendrian immersion 

$$i_{p/q}:\T\to L, i_{p/q}(t)=(pt,qt,0).$$

Our aim here is to study the behaviour of $\phi( i,i_{p/q})$ when $(p,q)\to \infty$. 

We first lift $i$ to an immersion $I:[a,b]\to \R^2\times \R$ of the form $I(t)=(x(t),y(t),\theta(t))$. By assumption we have $\dot{\theta}=-\frac{1}{2\pi}(x\dot y-y\dot x)$. For instance, lifting $i_{p/q}$ we get simply the map $I_{p/q}:t\mapsto(pt, qt,0)$. 

Let $r,s$ be integers such that $A=\begin{pmatrix} p& r \\ q& s\end{pmatrix}$ has determinant 1. Take $F_A:\R^2\to \R$ the function 
$$F_A(x,y)=\frac{1}{2\pi}(sx-ry)(qx-py)$$ 

A direct computation shows that this function satisfies $(m,n).I_{p/q}(t)=(pt+2\pi m,qt+2\pi n,F(pt+2\pi m,qt+2\pi n))$. We obtain from it the following formula:

\begin{equation}\label{F1}
\phi( i,i_{p/q}) =\sum_{a\le t \le b, qx(t)-py(t)\in 2\pi\Z} \delta_{\theta(t)-F(x(t),y(t))}.
\end{equation}

If we put $i=i_{0/1}$ this formula becomes $\phi(i_{0/1},i_{p/q})=\sum_{k=0}^{p-1}\delta_{2\pi\frac{rk^2}{p}}$. This measure is related to the usual Gauss sum in the sense that denoting by $q^*$ an inverse of $q$ mod $p$ we have:
$$\int e^{i\theta} d\phi(i_{0/1},i_{p/q})(\theta)=\sum_{k\in\Z/q\Z}\exp(2i\pi\frac{q^*k^2}{p}).$$

%
%
%

Suppose that $p_n=pn-r$ and $q_n=qn-s$. A Bézout matrix is given by $A_n=\begin{pmatrix}pn-r &p\\qn-s&q\end{pmatrix}$. Up to the action of $\SL_2(\Z)$, we can suppose that
 $p=s=1$ and $q=r=0$ in which case $F_{A_n}(x,y)=-\frac{y}{2\pi}(x+ny)$. We get from Equation \eqref{F1} the following formula for $\mu_n^\ell=\frac{1}{n}\int e^{i\ell\theta}d\phi(i,i_{pn/-1})(\theta)$:

\begin{equation}\label{F2}
\mu_n^\ell=\frac{1}{n}\sum_{\substack{ x(t)+ny(t)\in 2\pi \Z\\ a\le t\le b}} \exp\left(i\ell(\theta(t)+\frac{y(t)}{2\pi}(x(t)+ny(t)))\right).
\end{equation}

Taking $\ell=0$, we are simply counting the number of solutions of $x(t)+ny(t)\in 2\pi\Z$ for $t\in [a,b]$. Assuming that $y$ is monotonic, the number of solutions for $t\in [a,b]$ is asymptotic to $|y(b)-y(a)|$. Hence the asymptotic density of intersection points is $i^*|dy|$ and we get 

$$\lim_{n\to \infty}\mu_n^0=\int_a^b i^*|dy|.$$
To treat the case $\ell>0$, we need the following version of the Poisson formula: 
\begin{lemma} If $f,g:[a,b]\to \R$ are respectively $C^1$ and continuous and $f$ is piecewise monotonic, then if further $f(a),f(b)\notin 2\pi\Z$ we have 

$$\sum_{a\le t\le b, f(t)\in 2\pi\Z}g(t)=\frac{1}{2\pi}\sum_{k\in\Z}\int_a^b e^{-ikf(t)}|f'(t)|g(t)dt$$
\end{lemma}
Applying it here, we get 
$$\mu^\ell_n=\frac{1}{2\pi}\sum_{k\in\Z}\int_a^{b}e^{-ik(x+ny)+i\ell(\theta+\frac{y}{2\pi}(x+ny))}|\frac{\dot{x}}{n}+\dot y|dt$$
We apply a stationary phase expansion in this integral, the phase being $\Phi=-ky+ly^2/2\pi$ and its derivative being $\dot \Phi=(-k+ly/\pi)\dot y$. 
We find two types of critical points: the horizontal tangents $\dot y=0$ and the points of rational height $y=\pi \frac{k}{l}$. We observe that when $\dot{y}=0$ the amplitude 
is $O(\frac{1}{n})$ and hence these contributions can be neglected compared with the other ones, where $y=\pi \frac{k}{l}$. 

We compute $\ddot{\Phi}=\frac{l}{\pi}\dot{y}^2+(-k+ly/\pi)\ddot{y}=\frac{l}{\pi}\dot{y}^2$ and $\Phi=-\frac{\pi k^2}{2l}$. As $\ddot{\Phi}>0$, the stationary phase approximation gives 

$$\mu_n^l= \frac{1}{\sqrt{2n}} \sum_{y=\frac{\pi k}{l}}e^{-in\frac{k^2\pi}{2l}-i\frac{kx}{2}+i l\theta}+O(\frac{1}{n})$$

In order to give the final result, observe that the map $t\mapsto (t,\pi\frac{k}{l},\frac{kt}{2})$ defines a flat section of $L^\ell$ that we denote by $i_{1/0}^{k/\ell}$.


We can sum up the discussion by stating the following proposition.

\begin{proposition}\label{technique}
Let $i:\T\to L$ be a Legendrian immersion and suppose that $\pi\circ i$ is transverse to $i_{pn/-1}$ for $n$ large enough and to the circles of equation $y=\pi \xi$ for $\xi\in \Q$. 

Then writing $i(t)=(x(t),y(t),\theta(t))$ and $\mu_n^{\ell}=\frac{1}{n}\int e^{i\ell\theta}d\phi(i,i_{pn/-1})(\theta)$ we have for all $\ell>0$:

$$\mu_n^{\ell}=\frac{1}{\sqrt{2n}}\sum_{k\in\Z/2\ell\Z} \sum_{t\in \T, y(t)=\pi k/\ell} e^{-in\pi \frac{k^2}{2\ell}+i\phi\big(i(t),i_{1/0}^{k/l}(x(t))\big)}+O\Big(\frac{1}{n}\Big)$$
\end{proposition}

\subsection{Application to Chern-Simons invariants}\label{hypothesis}

Let $M$ be a 3-manifold with $\partial M=\T\times \T$. We assume that $X(M)$ is at most $1$-dimensional and that the restriction map $\Res:X(M)\to X(\partial M)$ is an immersion on the smooth part and map the singular points to non-torsion points. Then we know that $\Res(X(M))$ is transverse to $\T_{p/q}$ for all but a finite number of $p/q$, see \cite{marche-maurin}. 

Consider the projection map $\pi:\T^2\to X(\partial M)$ which is a 2-fold ramified covering. We may decompose $X(M)$ as a union of segments $[a_i,b_i]$ whose extremities contain all singular points. The restriction map $\Res$ can be lifted to $\T^2$ and the Chern-Simons invariant may be viewed as a map $\CS:[a_i,b_i]\to L$. Hence, we may apply it the results of Proposition \ref{technique} and obtain Theorem \ref{main}.

We may comment that the flat sections $i_{1/0}^{k/\ell}$ of $L^\ell$ over the line $y=\frac{\pi k}{\ell}$ induces through the quotient $(x,y,\theta)\sim (-x,-y,-\theta)$ a flat section of $L^\ell$ that we denoted  $\CS_{0/1}^{k/l}$ over the subvariety $\Res_{0/1}^{-1}(\frac{\pi k}{\ell})$.

\section{Chern-Simons invariants of coverings}

\subsection{General setting}
Beyond Dehn fillings, we can ask for the limit of the Chern-Simons measure of any sequence of 3-manifolds. A natural class to look at is the case of coverings of a same manifold $M$. Among that category, one can restrict to the family of cyclic coverings. One can even specify the problem to the following case. 

{\bf Question:} Let $p:M\to \T$ be a fibration over the circle and $M_n$ be the pull-back of the self-covering of $\T$ given by $z\mapsto z^n$. What is the asymptotic behaviour of $\mu_{M_n}$?

This problem can be formulated in the following way. Let $\Sigma$ be the fiber of $M$ and $f\in \Mod(\Sigma)$ be its monodromy. Any representation $\rho \in X(M)$ restricts to a representation $\Res(\rho)\in X(\Sigma)$ invariant by the action $f_*$ of $f$ on $X(\Sigma)$. Reciprocally, any irreducible representation $\rho\in X(\Sigma)$ fixed by $f_*$ correspond to two irreducible representations in $X(M)$. 

The Chern-Simons invariant corresponding to a fixed point may be computed in the following way: pick a path $\gamma:[0,1]\to X(\Sigma)$ joining the trivial representation to $\rho$ and consider the closed path obtained by composing $\gamma$ with $f(\gamma)$ in the opposite direction. Then its holonomy along $L$ is the Chern-Simons invariant of the corresponding representation.   

Understanding the asymptotic behaviour of $\mu_{M_n}$ consists in understanding the fixed points of $f_*^n$ on $X(\Sigma)$ and the distribution of Chern-Simons invariants of these fixed points, a problem which seems to be out of reach for the moment.

\subsection{Torus bundles over the circle}

In this elementary case, the computation can be done. Let $A\in \SL_2(\Z)$ act on $\R^2/\Z^2$. Its fixed points form a group $G_A=\{v\in \Q^2, Av=v\bmod \Z^2\}/\Z^2$. If $\tr(A)\ne 2$, which we suppose from now, $G_A$ is isomorphic to $\coker(A-\Id)$ and has cardinality $|\det(A-\Id)|$. 

Following the construction explained above, the phase is a map $f:G_A\to\Q/\Z$ given by $f([v])=\det(v,Av)\bmod \Z$. Hence, the measure we are trying to understand is the following:

$$\mu_A=\frac{1}{|\det(A-\Id)|}\sum_{v\in G_A}\delta_{2\pi\det(v,Av)}.$$

Consider the $\ell$-th moment $\mu_A^\ell$ of $\mu_A$. It is a kind of Gauss sum that can be computed explicitly. The map $f$ is a quadratic form on $G_A$ with values in $\Q/\Z$. Its associated bilinear form is $b(v,w)=\det(v,Aw)+\det(w,Av)=\det(v,(A-A^{-1})w)$. As $A+A^{-1}=\tr(A)\Id$ and $\det(A-\Id)=2-\tr(A)$ we get $b(v,w)= 2\det(v,(A-\Id)w)\mod \Z$. 
Hence, if $2\ell$ is invertible in $G_A$, then $\ell b$ is non-degenerate and standard arguments (see \cite{turaev} for instance) show that $|\mu_A^\ell|=|\det(A-\Id)|^{-1/2}$. 
Hence we still get the same kind of asymptotic behaviour for the Chern-Simons measure of the torus bundles over the circle.

\end{document}